\documentclass[a4paper,12pt]{article}
\usepackage{latexsym,amsmath,amsthm,amssymb}
\usepackage{a4wide}
\usepackage{hyperref}
\usepackage{marginnote}
\usepackage{color}
\usepackage{xcolor}
\usepackage{cite}
\hypersetup{
pdftitle={Adams hyperbolic}   
pdfauthor={Van Hoang Nguyen},
colorlinks = true,
linkcolor = magenta,
citecolor = blue,
}

\theoremstyle{plain}
\newtheorem{theorem}{Theorem}[section]

\newtheorem{proposition}[theorem]{Proposition}
\newtheorem{lemma}[theorem]{Lemma}

\theoremstyle{definition}

\theoremstyle{remark}

\renewcommand{\thefootnote}{\arabic{footnote}}

\def\R{\mathbb R}
\def\N{\mathbb N}

\def\al{\alpha}
\def\om{\omega}
\def\Om{\Omega}
\def\be{\beta}

\def\si{\sigma}
\def\lam{\lambda}
\def\ep{\epsilon}
\def\na{\nabla}
\def\pa{\partial}
\def\la{\langle} 
\def\ra{\rangle} 
\def\lt{\left}
\def\rt{\right}

\def\H{\mathbb H}

\def\leqs{<}
\def\geqs{>}

\numberwithin{equation}{section}

\title{The sharp Sobolev type inequalities in the Lorentz--Sobolev spaces in the hyperbolic spaces}
\author{Van Hoang Nguyen
}

\begin{document}
\maketitle


\renewcommand{\thefootnote}{}

\footnote{Email:  \href{mailto: Van Hoang Nguyen <vanhoang0610@yahoo.com>}{vanhoang0610@yahoo.com}.}

\footnote{2010 \emph{Mathematics Subject Classification\text}: 26D10, 46E35, 46E30, }

\footnote{\emph{Key words and phrases\text}: Poincar\'e inequality, Poincar\'e--Sobolev inequality, Moser--Trudinger inequality, Lorentz--Sobolev space, hyperbolic spaces.}

\renewcommand{\thefootnote}{\arabic{footnote}}
\setcounter{footnote}{0}

\begin{abstract}
Let $W^1L^{p,q}(\H^n)$, $1\leq q,p \leqs \infty$ denote the Lorentz--Sobolev spaces of order one in the hyperbolic spaces $\H^n$. Our aim in this paper is three-fold. First of all, we establish a sharp Poincar\'e inequality in $W^1L^{p,q}(\H^n)$ with $1\leq q \leq p$ which generalizes the result in \cite{NgoNguyenAMV} to the setting of Lorentz--Sobolev spaces. Second, we prove several sharp Poincar\'e--Sobolev type inequalities in $W^1L^{p,q}(\H^n)$ with $1\leq q \leq p \leqs n$ which generalize the results in \cite{NguyenPS2018} to the setting of Lorentz--Sobolev spaces. Finally, we provide the improved Moser--Trudinger type inequalities in $W^1L^{n,q}(\H^n)$ in the critical case $p= n$ with $1\leq q \leq n$ which generalize the results in \cite{NguyenMT2018} and improve the results in \cite{YangLi2019}. In the proof of the main results, we shall prove a P\'olya--Szeg\"o type principle in $W^1 L^{p,q}(\H^n)$ with $1\leq q \leq p$ which maybe is of independent interest.

\end{abstract}

\section{Introduction}
It is well known that the Sobolev type inequalities such as Poincar\'e inequality, Sobolev inequality, the Moser--Trudinger inequality, etc are the important and useful tools in many branches of mathematics such as Analysis, Geometry, Calculus of Variations, Partial Differential Equations, etc. In this paper, we study the Sobolev type inequalities in the Lorentz--Sobolev spaces defined in the hyperbolic spaces. Especially, we are interested in studying the sharp form of these inequalities. Let us start by recalling several related results in the setting of Euclidean spaces. Let $\Om$ be an open subset of $\R^n$, $n\geq 2$ and let $W^{1,p}_0(\Om)$ with $1\leq p \leqs \infty$ be the  the usual Sobolev space obtained by completion of $C_0^\infty(\Om)$ under the norm $\|\na u\|_{L^p(\Om)} =\Big(\int_\Om |\na u|^p dx\Big)^{\frac 1p}$. For bounded domain $\Om$, the Poincar\'e inequality asserts that
\begin{equation}\label{eq:PoincareE}
\int_\Om |\na u|^p dx \geq C_P(p,\Om)\int_\Om |u|^p dx,\quad u\in W^{1,p}_0(\Om)
\end{equation}
where $C_P(p,\Om)$ denotes the possibly smallest constant for which \eqref{eq:PoincareE} holds. It is an interesting problem is to determine the explicit value of $C_P(p,\Om)$. A few results for $C_P(2,\Om)$ is known when $\Om$ is the unit ball or a convex domain (see, e.g., \cite{Acosta,Kuznetsov,Payne}). For general domain $\Om$ and arbitrary $p$, determining the constant $C_P(p,\Om)$ is a hard task since the value of $C_P(p,\Om)$ depends on $p$ and the geometry of $\Om$. 

In the hyperbolic spaces, the Poincar\'e inequality was proved by Tataru \cite{Tataru2001}. Let $\H^n$ denote the hyperbolic space of dimension $n\geq 2$ that is a Riemannian manifold with a Riemannian metric $g$ such that its sectional curvature is $-1$ (see Section \S2 below for more details). It was proved in \cite{Tataru2001} that 
\begin{equation}\label{eq:Tataru}
\int_{\H^n} |\na_g u|_g^p dV_g \geq C \int_{\H^n} |u|^p dV_g,\quad u\in C_0^\infty(\H^n)
\end{equation}
for some constant $C \geqs 0$, where $\na_g, |\cdot|_g$ and $dV_g$ denote the hyperbolic gradient, hyperbolic length of a vector field and the volume element in $\H^n$ with respect to the Riemannian metric $g$, respectively. Finding the sharp value of constant $C$ in \eqref{eq:Tataru} is an interesting question. Mancini and Sandeep \cite{ManciniSandeep2008} show that the sharp value of the constant $C$ in \eqref{eq:Tataru} is $(N-1)^2/4$ when $p =2$. For arbitrary $p$, it was proved by Ngo and the author \cite{NgoNguyenAMV} that the sharp constant $C$ in \eqref{eq:Tataru} is $(N-1)^p/p^p$ (see \cite{BerchioNA} for another proof). The first aim in this paper is to generalize the sharp Poincar\'e inequality in $\H^n$ to the more general context of the Lorentz--Sobolev spaces defined in $\H^n$. For $1\leq p, q \leqs \infty$, we denote by $W^1 L^{p,q}(\H^n)$ the Lorentz--Sobolev space of order one in $\H^n$ which is the completion of $C_0^\infty(\H^n)$ under the quasi-norm $\|\na_g u\|_{p,q} := \| |\na_g u|_g\|_{p,q}$ (see Section \S2 below for the definition of Lorentz spaces and the Lorentz semi-norm $\|\cdot\|_{p,q}$. The first main result in this paper is as follows

\begin{theorem}\label{Poincare}
Let $n \geq 2$ and $1 \leqs q \leq p \leqs \infty$. Then it holds 
\begin{equation}\label{eq:PoincareLorentz}
\|\na_g u\|_{p,q}^q \geq \lt(\frac{n-1}p\rt)^q \|u\|_{p,q}^q,\quad\forall\, u\in W^1 L^{p,q}(\H^n).
\end{equation}
Furthermore, the constant $(n-1)^q/p^q$ in \eqref{eq:PoincareLorentz} is sharp and never attained in $W^1 L^{p,q}(\H^n)$.
\end{theorem}

Theorem \ref{Poincare} covers the result in \cite{NgoNguyenAMV} for $m =1$ and in \cite{BerchioNA} in the special case $p= q$. Notice that the sharp constant $(n-1)^p/p^p$ in \eqref{eq:Tataru} is never archived in $W^{1,p}(\H^n)$. This leaves a room for several improvements of the inequality \eqref{eq:Tataru} with sharp constant. Notice that the non achievement of sharp constant does not always imply improvement (e.g., Hardy operator in the Euclidean space $\R^n, n \geq 2$). However, in the hyperbolic space, the operator $-\Delta_{g,p}-(\frac{n-1}p)^p =-\text{\rm div}(|\na_g \cdot|_g^{p-2} \na_g \cdot)-(\frac{n-1}p)^p$ is sub-critical, hence improvement is possible. For examples, the reader can consult the papers \cite{BerchioNA,BerchioJFA} for the improvements of \eqref{eq:Tataru} by adding the remainder terms concerning to Hardy weights, i.e., the inequalities of the form 
\[
\int_{\H^n} |\na_g u|_g^p dV_g -\lt(\frac{n-1}p\rt)^p \int_{\H^n} |u|^p dV_g \geq \int_{\H^n} W |u|^p dV_g
\]
where $W \geq 0$ is the weight $W$ satisfying some appropriate conditions. For the case $p =2$, Mancini and Sandeep \cite{ManciniSandeep2008} proved the following Poincar\'e--Sobolev inequality which provides another improvement of \eqref{eq:Tataru}
\begin{equation}\label{eq:MS}
\int_{\H^n} |\na_g u|_g^2 dV_g -\lt(\frac{n-1}2\rt)^2 \int_{\H^n} |u|^2 dV_g \geq C \lt(\int_{\H^n} |u|^{\frac{2n}{n-2}} dV_g\rt)^{\frac{n-2}n},
\end{equation}
for some $C \geqs 0$. The inequality \eqref{eq:MS} is equivalent to a Hardy--Sobolev--Maz'ja inequality in the half space (see \cite[Section $2.1.6$]{Maz'ja}). Let $C_n$ denote the possibly smallest constant $C$ in the right hand side of \eqref{eq:MS}. In \cite{Tertikas}, Tertikas and Tintarev proved for $n\geq 4$ that $C_n \leqs S_n$ and $C_n$ is attained, where $S_n$ denotes the best constant in the $L_2$ Sobolev inequality in $\R^n$ (see \cite{Aubin1976,Talenti1976}). More surprising in three dimensional cases, it was show that $C_3 = S_3$ and is not attained (see \cite{BenguriaFrankLoss}). For arbitrary $p\not=2$, the author established in \cite{NguyenPS2018} the following $L^p$ Poincar\'e--Sobolev inequality in $\H^n$ 
\begin{equation}\label{eq:NguyenPS}
\int_{\H^n} |\na_g u|_g^p dV_g -\lt(\frac{n-1}p\rt)^p \int_{\H^n} |u|^p dV_g \geq S_{n,p}^p \lt(\int_{\H^n} |u|^{\frac{np}{n-p}} dV_g\rt)^{\frac{n-p}n}
\end{equation}
for $n\geq 4$ and $\frac{2n}{n-1} \leq p  \leqs n$ where $S_{n,p}, 1\leq p \leqs n$ is the sharp constant in the $L_p$ Sobolev inequality in $\R^n$
\begin{equation}\label{eq:SobolevRn}
\int_{\R^n} |\na u |^p dx \geq S_{n,p}^p \lt(\int_{\R^n} |u|^{\frac{np}{n-p}} dx\rt)^{\frac{n-p}{n}},\quad u\in W^{1,p}(\R^n).
\end{equation}
The sharp constant $S_{n,p}$ in \eqref{eq:SobolevRn} was found independently by Talenti \cite{Talenti1976} and Aubin \cite{Aubin1976} (another proof of the sharp Sobolev inequality \eqref{eq:SobolevRn} via the optimal transport of measure could be found in \cite{CNV}). The second aim in this paper is to establish an analogue of the Poincar\'e--Sobolev inequality \eqref{eq:NguyenPS} in the Lorentz--Sobolev space $W^1 L^{p,q}(\H^n)$ which provide an improvement of the sharp Poincar\'e inequality from Theorem \ref{Poincare}. Before stating our next results, let us recall that the sharp Sobolev inequality in the Lorentz--Sobolev space $W^1 L^{p,q}(\R^n)$ (the completion of $C_0^\infty(\R^n)$ under the Lorentz quasi-norm $\|\na \cdot \|_{p,q}$) was proved by Alvino \cite{Alvino1977}. More precisely, Alvino has show for $1 \leq q \leq p \leqs n$ that
\begin{equation}\label{eq:Alvino}
\|\na u\|_{p,q}^q \geq \lt(\frac{n-p}p\rt)^q \sigma_n^{\frac qn} \|u\|_{p^*,q}^q,\quad u \in W^1 L^{p,q}(\R^n)
\end{equation}
where $\sigma_n$ is the volume of the unit ball in $\R^n$. The inequality \eqref{eq:Alvino} is sharp, and in the case $q =p$ it is equivalent to the famous Hardy inequality (see \cite{Davies,Kufner} for the history of the Hardy inequality). In recent paper, Cassani, Ruf and Tarsi \cite{CassaniRufTarsi2018} extend the inequality \eqref{eq:Alvino} to the case $p \leqs q \leq \infty$.

Our next main result reads as follows

\begin{theorem}\label{PoincareSobolev}
Let $n\geq 4$ and $\frac{2n}{n-1}\leq q \leq p \leqs n$. Then, for any $q\leq l \leq \frac{nq}{n-p}$ we have
\begin{equation}\label{eq:PSLorentz1}
\|\na_g u\|_{p,q}^q - \lt(\frac{n-1}p\rt)^q \|u\|_{p,q}^q \geq  S_{n,p,q,l}^q \|u\|_{p^*,l}^q,\quad\forall\, u\in W^1 L^{p,q}(\H^n),
\end{equation}
where $p^* = \frac{np}{n-p}$, and
\[
S_{n,p,q,l} =\begin{cases}
\lt[n^{1-\frac ql} \sigma_n^{\frac qn} \Big(\frac{(n-p)(l-q)}{qp}\Big)^{q+\frac ql -1} S\Big(\frac{lq}{l-p},q\Big)\rt]^{\frac 1q} &\mbox{if $q \leqs l \leq \frac{nq}{n-p}$,}\\
\frac{n-p}p \sigma_n^{\frac 1n}&\mbox{if $l =q$,}
\end{cases}
\]
with $S\big(\frac{lq}{l-p},q\Big)$ being the sharp constant in the Sobolev inequality with fractional dimension appearing in the Lemma \ref{Sobolevfract} below. Moreover, the constant $S_{n,p,q,l}$ on the right hand side of \eqref{eq:PSLorentz1} is sharp.
\end{theorem}
Notice that the inequality \eqref{eq:PSLorentz1} reduces to the Poincar\'e--Sobolev inequality \eqref{eq:NguyenPS} in the case $q =p$ and $l = \frac{np}{n-q}$. We can readily check that the constant $S_{n,p,q,l}$ in \eqref{eq:PSLorentz1} is sharp by the fact that the inequality 
\[
\|\na u\|_{p,q}^q \geq S_{n,p,q,l}^q \|u\|_{p^*,l}^q
\]
in $\R^n$ is sharp and scaling invariant. So it is still sharp constant for the same inequality in the ball $B_r(0)\subset \R^n$ of radius $r$ and center at origin. For $r$ sufficient small, we have $V_g \sim 2^n \mathcal L$ in $B_r(0)$ here $\mathcal L$ denotes Lebesgue's measure in $\R^n$, and $|\na_g \cdot|_g \sim \frac12 |\na \cdot|$ in $B_r(0)$. From these facts, we have 
\[
\|\na_g u\|_{p,q;\H^n}^q \sim 2^{\frac{n-p}p q} \|\na u\|_{p,q;\R^n}^q, \quad\text{\rm and}\quad \|u\|_{p^*,l;\H^n}^q \sim 2^{\frac{n-p}p q}\|u\|_{p^*,l;\R^n}^q
\]
for function $u$ with support in $B_r(0)$, here the indexes with $\H^n$ and $\R^n$ denote the Lorentz quasi-norm in $\H^n$ and $\R^n$ respectively. This implies the sharpness of $S_{n,p,q,l}$ in the Lorentz--Sobolev type inequality in $W^1 L^{p,q}(\H^n)$
\[
\|\na_g u\|_{p,q}^q \geq S_{n,p,q,l}^q \|u\|_{p^*,l}^q.
\]
So, $S_{n,p,q,l}$ is also sharp in \eqref{eq:PSLorentz1}. Thus, the inequality \eqref{eq:PSLorentz1} improves not only the sharp Poincar\'e inequality \eqref{eq:PoincareLorentz} but also the sharp Lorentz--Sobolev inequality above in the Lorentz--Sobolev space $W^1 L^{p,q}(\H^n)$.

We next move to the critical case $p =n$. Let us recall that for $1\leq p\leqs n$ and a bounded domain $\Om$ in the $\R^n$ we have the following Sobolev embedding $W^{1,p}_0(\Om) \hookrightarrow L^{p^*}(\Om)$. However, in the critical case $p=n$ (i.e., $p^* =\infty$), we don't have the embedding $W^{1,n}_0(\Om) \hookrightarrow L^{\infty}(\Om)$. Instead of the Sobolev inequality in this critical case, we have the Moser--Trudinger inequality. The Moser--Trudinger inequality was independently proved by Trudinger \cite{Trudinger67}, Yudovic \cite{Yudovic1961} and Pohozaev \cite{Pohozaev1965}. It was sharpened by Moser \cite{Moser70} in the following form
\begin{equation}\label{eq:Moser}
\sup_{u \in W^{1,n}_0(\Om), \, \int_\Om |\na u|^n dx \leq 1} \int_\Om e^{\al |u|^{\frac n{n-1}}} dx \leqs \infty,
\end{equation}
for any $\al \leq \al_n: = n \om_{n-1}^{\frac1{n-1}}$ where $\om_{n-1}$ denotes the surface area of the unit sphere in $\R^n$. Furthermore, if $\al \geqs \al_n$ then the supremum in \eqref{eq:Moser} becomes infinite though all integrals are still finite. Since its appearance, there have been many generalizations of the Moser--Trudinger inequality \eqref{eq:Moser} in many directions (see, e.g., \cite{Adams,AdimurthiDruet2004,Tintarev2014,YangSuKong,WangYe2012,AdimurthiSandeep2007,ManciniSandeep2010,AdimurthiTintarev2010,ManciniSandeepTintarev2013,CohnLu,CohnLu1,Nguyenimproved,LamLuHei,deOliveira}). Concerning to the extremals of \eqref{eq:Moser}, we refer the readers to the papers \cite{Carleson86,Flucher92,Lin96}. The Moser--Trudinger inequality \eqref{eq:Moser} was extended to unbounded domain $\Om\subset\R^n$ by Adachi and Tanaka \cite{Adachi00} in the following scaling invariant form
\begin{equation}\label{eq:AT}
\sup_{u\in W^{1,n}(\R^n), \int_{\R^n} |\na u|^n dx \leq 1} \frac{1}{\|u\|_{L^n(\R^n)}^n}\int_{\R^n} \Phi(\al |u|^{\frac n{n-1}}) dx < \infty
\end{equation}
for any $\alpha < \alpha_n$ where $\Phi(t) = e^t - \sum_{j=0}^{n-2} t^j/j!$ (i.e., the truncation of the exponential function). Notice that the critical exponent $\al_n$ is not allowed in \eqref{eq:AT}. Later, Ruf \cite{Ruf2005} and Li and Ruf \cite{LiRuf2008} established the sharp Moser--Trudinger inequality in unbounded domains under the full norm in $W^{1,n}(\R^n)$ which allows the critical exponent $\al_n$ as follows
\begin{equation}\label{eq:LiRuf}
\sup_{u\in W^{1,n}(\R^n), \|u\|_{W^{1,n}(\R^n)}^n:=\int_{\R^n}( |\na u|^n+ |u|^n) dx  \leq 1} \int_{\R^n} \Phi(\al |u|^{\frac n{n-1}}) dx < \infty
\end{equation}
for any $\alpha \leq \alpha_n$. Furthermore, the existence of maximizers for \eqref{eq:LiRuf} also is addressed in \cite{Ruf2005,LiRuf2008}. A singular version of \eqref{eq:LiRuf} was given by Adimurthi and Yang \cite{AdimurthiYang2010}. Another proof of \eqref{eq:LiRuf} without using rearrangement arguments was provided by Lam and Lu \cite{LamLunewapproach}. The Moser--Trudinger inequality in the Lorentz spaces $W^{1}L^{n,q}(\Omega), 1< q < \infty$ for a bounded domain $\Om$ was established by Alvino, Ferone and Trombetti \cite{Alvino1996} in the following form
\begin{equation}\label{eq:MoserLorentz}
\sup_{u\in W^1 L^{n,q}(\Om), \, \|\na u\|_{n,q} \leq 1} \int_{\Om} e^{\al |u|^{\frac q{q-1}}} dx < \infty
\end{equation}
for any $\al \leq \al_{n,q} :=(n^{\frac{n-1}n} \om_{n-1}^{\frac1n})^{\frac q{q-1}}$ if $1< q < \infty$. Notice that the constant $\al_{n,q}$ is sharp in \eqref{eq:MoserLorentz} in the sense that the supremum will become infinite if $\al >\al_{n,q}$. For unbounded domains in $\R^n$, the Moser--Trudinger inequality was proved by Cassani and Tarsi \cite{CassaniTarsi2009} (see Theorem $1$ and Theorem $2$ in \cite{CassaniTarsi2009}). In \cite{LuTang2016}, Lu and Tang proved several sharp singular Moser--Trudinger inequalities in the Lorentz--Sobolev spaces which generalize the results in \cite{Alvino1996,CassaniTarsi2009} to the singular weights.

In the hyperbolic space $\H^n$, the Moser--Trudinger inequality was firstly proved by Mancini and Sandeep \cite{ManciniSandeep2010} in the dimension $n =2$ (another proof of this result was given by Adimurthi and Tintarev \cite{AdimurthiTintarev2010}) and by Mancini, Sandeep and Tintarev \cite{ManciniSandeepTintarev2013} in higher dimension $n\geq 3$ (see \cite{FontanaMorpurgo2020} for an alternative proof)
\begin{equation}\label{eq:MThyperbolic}
\sup_{u\in W^{1,n}(\H^n),\, \int_{\H^n} |\na_g u|_g^n dV_g \leq 1} \int_{\H^n} \Phi(\al_n |u|^{\frac n{n-1}}) dV_g < \infty.
\end{equation}
Lu and Tang \cite{LuTang2013} also established the sharp singular Moser--Trudinger inequality under the conditions $\|\na u\|_{L^n(\H^n)}^n + \tau \|u\|_{L^n(\H^n)}^n \leq 1$ for any $\tau >0$ (see Theorem $1.4$ in \cite{LuTang2013}). In \cite{NguyenMT2018}, the author improves the inequality \eqref{eq:MThyperbolic} by proving the following inequality 
\begin{equation}\label{eq:NguyenMT}
\sup_{u\in W^{1,n}(\H^n),\, \int_{\H^n} |\na_g u|_g^n dV_g - \lam \int_{\H^n} |u|^n dV_g \leq 1} \int_{\H^n} \Phi(\al_n |u|^{\frac n{n-1}}) dV_g < \infty,
\end{equation}
for any $\lambda < (\frac{n-1}n)^n$.

To our knowledge, much less is known about the sharp Trudinger--Moser inequalities under Lorentz norm on complete noncompact Riemannian manifolds except Euclidean spaces. Recently, Yang and Li \cite{YangLi2019} proves a sharp Moser--Trudinger inequality in the Lorentz--Sobolev spaces defined in the hyperbolic spaces. More precisely, their result (\cite[Theorem $1.6$]{YangLi2019}) states that for $1\leqs q \leqs \infty$  it holds
\begin{equation}\label{eq:YangLi}
\sup_{u\in W^1L^{n,q}(\H^n),\, \|\na_g u\|_{n,q} \leq 1} \int_{\H^n} \Phi_q(\al_{n,q} |u|^{\frac q{q-1}}) dV_g \leqs \infty,
\end{equation}
where
\[
\Phi_q(t) =e^t - \sum_{j=0}^{j_q -2} \frac{t^j}{j!},\quad \text{\rm where}\,\, j_q = \min\{j\in \N\, :\, j \geq 1+ n(q-1)/q\}.
\]
The third (and last) aim in this paper is to establish an analogue of \eqref{eq:NguyenMT} in the Lorentz--Sobolev space $W^1 L^{n,q}(\H^n)$ and hence give an improvement of the inequality of Yang and Li. Our next result provides such an analogue of \eqref{eq:NguyenMT} and is stated as follows.

\begin{theorem}\label{3rdMT}
Let $n\geq 4$ and $\frac{2n}{n-1} \leq q \leq n$. Then we have
\begin{equation}\label{eq:improvedMTLorentz}
\sup_{u\in W^1L^{n,q}(\H^n),\, \|\na_g u\|_{n,q}^q -\lam \|u\|_{n,q}^q \leq 1} \int_{\H^n} \Phi_q\big(\al_{n,q} |u|^{\frac q{q-1}}\big) dV_g \leqs \infty,
\end{equation}
for any $\lam \leqs (\frac{n-1}n)^q$.
\end{theorem}
Notice that the inequality \eqref{eq:NguyenMT} is a special case of \eqref{eq:improvedMTLorentz} corresponding to the case $q =n$. Obviously, the inequality \eqref{eq:improvedMTLorentz} improves the result of Yang and Li \eqref{eq:YangLi}. However, comparing with the result of Yang and Li, we need impose an extra condition $q \leq n$. This condition is necessary to apply the rearrangement argument. We shall discuss, in details, about this point below.

We conclude this introduction by some comments on the proof of our main results (i.e., Theorems \ref{Poincare}, \ref{PoincareSobolev}, \ref{3rdMT}). To prove our main results, we adopt and develop the approach in \cite{NgoNguyenAMV,NguyenMT2018,NguyenPS2018}. Our approach heavily relies on the rearrangement argument applied to the hyperbolic spaces $\H^n$. In order to apply this argument, we need impose the condition $q \leq p$ (or $q \leq n$) in our main results. In fact, under this extra condition, we shall prove a P\'olya--Szeg\"o type principle in the Lorentz--Sobolev spaces $W^1L^{p,q}(\H^n)$ (see Theorem \ref{PS} below)
\[
\|\na_g u^\sharp\|_{p,q} \leq \|\na_g u\|_{p,q},\quad u\in W^1L^{p,q}(\H^n)
\]
for $1\leq q \leq p$, where $u^\sharp$ denotes the radially symmetric non-increasing rearrangement function of $u$ (see Section \S2 below for details). By this result, we can reduce the proof of Theorems \ref{Poincare}, \ref{PoincareSobolev} and \ref{3rdMT} to radially symmetric non-increasing functions in $\H^n$. For such functions, we prove a key estimate in Proposition \ref{keyprop} which is crucial in the proofs of Theorems \ref{PoincareSobolev} and \ref{3rdMT}. The detail proofs of these theorems are given in Sections \S3, \S4 and \S5 below. Finally, it is worth to mention here that the main results in this paper are recently extended to the higher order Lorentz--Sobolev spaces defined in the hyperbolic spaces by the author in \cite{Nguyen2020,Nguyen2020Adams}.

The rest of this paper is organized as follows. In the next section, we give some basses on the hyperbolic spaces $\H^n$ and the Lorentz--Sobolev spaces $W^1 L^{p,q}(\H^n)$, and prove the P\'olya--Szeg\"o type principle in $W^1 L^{p,q}(\H^n)$. The proof of Theorem \ref{Poincare} is given in Section \S3. Theorem \ref{PoincareSobolev} is proved in Section \S4. Finally , we prove Theorem \ref{3rdMT} in Section \S5.

\section{Preliminaries}
We start this section by briefly recalling some basis facts on the hyperbolic spaces and the Lorentz--Sobolev space defined in the hyperbolic spaces. Let $n\geq 2$, a hyperbolic space of dimension $n$ (denoted by $\H^n$) is a complete , simply connected Riemannian manifold having constant sectional curvature $-1$. There are several models for the hyperbolic space $\H^n$ such as the half-space model, the hyperboloid (or Lorentz) model and the Poincar\'e ball model. Notice that all these models are Riemannian isometry. In this paper, we are interested in the Poincar\'e ball model of the hyperbolic space since this model is very useful for questions involving rotational symmetry. In the Poincar\'e ball model, the hyperbolic space $\H^n$ is the open unit ball $B_n\subset \R^n$ equipped with the Riemannian metric
\[
g(x) = \Big(\frac2{1- |x|^2}\Big)^2 dx \otimes dx.
\]
The volume element of $\H^n$ with respect to the metric $g$ is given by
\[
dV_g(x) = \Big(\frac 2{1 -|x|^2}\Big)^n dx,
\]
where $dx$ is the usual Lebesgue measure in $\R^n$. For $x \in B_n$, let $d(0,x)$ denote the geodesic distance between $x$ and the origin, then we have $d(0,x) = \ln (1+|x|)/(1 -|x|)$. For $\rho \geqs 0$, $B(0,\rho)$ denote the geodesic ball with center at origin and radius $\rho$. If we denote by $\na$ and $\Delta$ the Euclidean gradient and Euclidean Laplacian, respectively as well as $\la \cdot, \cdot\ra$ the standard scalar product in $\R^n$, then the hyperbolic gradient $\na_g$ and the Laplace--Beltrami operator $\Delta_g$ in $\H^n$ with respect to metric $g$ are given by
\[
\na_g = \Big(\frac{1 -|x|^2}2\Big)^2 \na,\quad \Delta_g = \Big(\frac{1 -|x|^2}2\Big)^2 \Delta + (n-2) \Big(\frac{1 -|x|^2}2\Big)\la x, \na \ra,
\]
respectively. For a function $u$, we shall denote $\sqrt{g(\na_g u, \na_g u)}$ by $|\na_g u|_g$ for simplifying the notation. Finally, for a radial function $u$ (i.e., the function depends only on $d(0,x)$) we have the following polar coordinate formula
\begin{equation}\label{eq:polar}
\int_{\H^n} u(x) dx = n \sigma_n \int_0^\infty u(\rho) \sinh^{n-1}(\rho)\,  d\rho.
\end{equation}

It is now known that the symmetrization argument works well in the setting of the hyperbolic. It is the key tool in the proof of several important inequalities such as the Poincar\'e inequality, the Sobolev inequality, the Moser--Trudinger inequality in $\H^n$. We shall see that this argument is also the key tool to establish the main results in the present paper. Let us recall some facts about the rearrangement argument in the hyperbolic space $\H^n$. A measurable function $u:\H^n \to \R$ is called vanishing at the infinity if for any $t >0$ the set $\{|u| > t\}$ has finite $V_g-$measure, i.e.,
\[
V_g(\{|u|> t\}) = \int_{\{|u|> t\}} dV_g < \infty.
\]
For such a function $u$, its distribution function is defined by
\[
\mu_u(t) = V_g( \{|u|> t\}).
\]
Notice that $t \to \mu_u(t)$ is non-increasing and right-continuous. The non-increasing rearrangement function $u^*$ of $u$ is defined by
\[
u^*(t) = \sup\{s > 0\, :\, \mu_u(s) > t\}.
\] 
The non-increasing, spherical symmetry, rearrangement function $u^\sharp$ of $u$ is defined by
\[
u^\sharp(x) = u^*(V_g(B(0,d(0,x)))),\quad x \in \H^n.
\]
It is well-known that $u$ and $u^\sharp$ have the same non-increasing rearrangement function (which is $u^*$). Finally, the maximal function $u^{**}$ of $u^*$ is defined by
\[
u^{**}(t) = \frac1t \int_0^t u^*(s) ds.
\]
Evidently, $u^*(t) \leq u^{**}(t)$.

For $1\leq p, q < \infty$, the Lorentz space $L^{p,q}(\H^n)$ is defined as the set of all measurable function $u: \H^n \to \R$ satisfying
\[
\|u\|_{L^{p,q}(\H^n)}: = \lt(\int_0^\infty \lt(t^{\frac1p} u^*(t)\rt)^q \frac{dt}t\rt)^{\frac1q} < \infty.
\]
It is clear that $L^{p,p}(\H^n) = L^p(\H^n)$. Moreover, the Lorentz spaces are monotone with respect to second exponent, namely
\[
L^{p,q_1}(\H^n) \subsetneq L^{p,q_2}(\H^n),\quad 1\leq q_1 < q_2 < \infty.
\]
The functional $ u\to \|u\|_{L^{p,q}(\H^n)}$ is not a norm in $L^{p,q}(\H^n)$ except the case $q \leq p$ (see \cite[Chapter $4$, Theorem $4.3$]{Bennett}). In general, it is a quasi-norm which  turns out to be equivalent to the norm obtained replacing $u^*$ by its maximal function $u^{**}$ in the definition of $\|\cdot\|_{L^{p,q}(\H^n)}$. Moreover, as a consequence of Hardy inequality, we have
\begin{proposition}\label{hardy}
Given $p\in (1,\infty)$ and $q \in [1,\infty)$. Then for any function $u \in L^{p,q}(\H^n)$ it holds 
\begin{equation}\label{eq:Hardy}
\lt(\int_0^\infty \lt(t^{\frac1p} u^{**}(t)\rt)^q \frac{dt}t\rt)^{\frac1q} \leq \frac p{p-1} \lt(\int_0^\infty \lt(t^{\frac1p} u^*(t)\rt)^q \frac{dt}t\rt)^{\frac1q} = \frac p{p-1} \|u\|_{L^{p,q}(\H^n)}.
\end{equation}
\end{proposition}
For $1\leq p, q \leqs \infty$, we define the first order Lorentz--Sobolev space $W^1L^{p,q}(\H^n)$ by taking the completion of $C_0^\infty(\H^n)$ under the quasi-norm
\[
\|\na_g u\|_{p,q} := \| |\na_g u|_g\|_{p,q}.
\]
It is obvious that $W^1L^{p,p}(\H^n) = W^{1,p}(\H^n)$ the first order Sobolev space in $\H^n$. The P\'olya--Szeg\"o principle in the hyperbolic spaces asserts that if $u\in W^{1,p}(\H^n)$ then $u^\sharp \in W^{1,p}(\H^n)$ and 
$$\int_{\H^n} |\na_g u^\sharp|_g^p dV_g \leq \int_{\H^n} |\na_g u|_g^p dV_g.$$
This principle is very useful to find the sharp constant in several inequalities concerning to the $L^p$ norm of hyperbolic gradient. In the next result, we extend the P\'olya--Szeg\"o principle to the Lorenz--Sobolev spaces $W^1L^{p,q}(\H^n)$.
\begin{theorem}\label{PS}
Let $n\geq 2$, $1\leq q \leq p \leqs \infty$ and $u\in W^{1}L^{p,q}(\H^n)$. Then $u^\sharp \in W^{1}L^{p,q}(\H^n)$ and 
$$\|\na_g u^\sharp\|_{p,q} \leq \|\na_g u\|_{p,q}.$$
\end{theorem}
\begin{proof}
Since $(|u|)^\sharp = u^\sharp$ and $|\na_g |u||_g \leq |\na_g u|_g$, hence it is enough to prove Theorem \ref{PS} for nonnegative function $u\in W^1L^{p,q}(\H^n)$. Moreover, by the density, it is sufficient to assume that $u \in C_0^\infty(\H^n)$.
 
Let $u$ be a nonnegative function in $C_0^\infty(\H^n)$ with support $\Om\subset \H^n$. Following \cite{Alvino1989} we consider a function $U$ built from $|\na_g u|_g$ on the level sets of $u$, i.e., 
\begin{equation}\label{eq:builtU}
\int_{\{u > t\}} |\na_g u|_g dV_g = \int_0^{V_g(\{u> t\})} U(s) ds,
\end{equation}
for any $t >0$. Notice that $U \prec |\na_g u|_g$ in the sense that
\[
\int_0^t U^*(s) ds \leq \int_0^t (|\na_g u|_g)^*(s) ds,\quad \forall\, t \in [0, V_g(\Om)),
\]
and 
\[
\int_0^{V_g(\Om)} U^*(s) ds \leq \int_0^{V_g(\Om)} (|\na_g u|_g)^*(s) ds.
\]
From \eqref{eq:builtU}, we have by differentiating in $t$,
\[
-U(\mu_u(t)) \mu_u'(t) = \int_{\{u=t\}} \lt(\frac 2{1-|x|^2}\rt)^{n-1} d\mathcal H^{n-1}(x).
\]
Let us define the function 
$$\Psi(r)= V_g(B(0,r)) = n \sigma_n \int_0^r \sinh^{n-1}(s) ds, r\geq 0.$$
Let $F$ denote the inverse function of $\Psi$, i.e., $\Psi(F(r)) =r$ for any $r \geq 0$. If we denote by $\rho(t)$ the radius of the geodesic ball in $\H^n$ with volume equal to $\mu_u(t)$, then $\rho(t) = F(\mu_u(t))$. By isoperimertric inequality \cite{Bogelein}, we have
\begin{align*}
-U(\mu_u(t)) \mu_u'(t) &= \int_{\{u=t\}} \lt(\frac 2{1-|x|^2}\rt)^{n-1} d\mathcal H^{n-1}(x)\\
&\geq \int_{\pa B(0,\rho(t))} \lt(\frac 2{1-|x|^2}\rt)^{n-1} d\mathcal H^{n-1}(x)\\
&= n \sigma_n \sinh^{n-1} (F(\mu_u(t))).
\end{align*}
Consequently, we get
\[
(-u^*)'(s) \leq \frac{U(s)}{n \sigma_n \sinh^{n-1}(F(s))}, \quad s\in (0, V_g(\Om)).
\]
Define the function 
\[
v(x) = \int_{V_g(B(0,d(0,x))}^{V_g(\Om)} \frac{U(s)}{n \sigma_n \sinh^{n-1}(F(s))} ds,
\]
for $x \in \Om^\sharp$, where $\Om^\sharp$ denotes the geodesic ball in $\H^n$ with volume $V_g(\Om)$. Extending $v =0$ for $x\in \H^n \setminus \Om^\sharp$. We have $u^\sharp(x) \leq v(x)$ for any $x \in \H^n$. We have
\[
|\na_g u^\sharp(x)|_g \leq |\na_g v(x)|_g = U(V_g(B(0,d(x)))) \prec |\na_g u(x)|_g.
\]
Consequently, we get $(|\na_g u^\sharp|_g)^* \prec (|\na_g u|_g)^*$ in $(0,V_g(\Om))$. Since $q \leq p$, then Proposition $1$ in \cite{CassaniTarsi2009} (or Corollary $2.1$ in \cite{Alvino1989}), we have
\[
\int_0^{V_g(\Om)} ((|\na_g u^\sharp|_g)^*(t))^q t^{\frac qp -1} dt \leq \int_0^{V_g(\Om)} (|\na_g u|_g)^*(t) ((|\na_g u^\sharp|_g)^*(t))^{q-1} t^{\frac qp -1} dt.
\]
Applying H\"older inequality, we get $ \|\na_g u^\sharp\|_{L^{p,q}(\H^n)} \leq \|\na_g u\|_{L^{p,q}(\H^n)}$ as desired.
\end{proof}

The next proposition is the key in the proof of Theorem \ref{PoincareSobolev} and Theorem \ref{3rdMT}.

\begin{proposition}\label{keyprop}
Given $n\geq 3$, $1< p < \infty$ and $\frac{2n}{n-1} \leq q \leq p$. Let $u\in W^{1}L^{p,q}(\H^n)$ be a radially symmetric non-increasing function, and $u^*$ be its non-increasing rearrangement function. Define $v(r) =u^*(\sigma_n r^n)$, $r\geq 0$. Then it holds
\begin{equation}\label{eq:keyestimate}
\|\na_g u\|_{p,q}^q - \lt(\frac{n-1}p\rt)^q \|u\|_{p,q}^q \geq n \sigma_n^{\frac qp}\int_0^\infty |v'(r)|^q r^{\frac{nq}p -1} dr.
\end{equation}
\end{proposition}
\begin{proof}
Let $U$ be the function built from $|\na_g u|_g$ on the level sets of $u$. Since $u$ is radially symmetric, then we have
\[
U(V_g(B(0,d(0,x)))) = |\na_g u(x)|_g.
\]
Since $q \leq p$, by Hardy--Littlewood inequality, it holds that
\begin{equation}\label{eq:gradientLorentz}
\int_0^\infty U(t)^q t^{\frac qp -1} dt \leq \int_0^\infty (U^*(t))^q t^{\frac qp -1} dt = \int_0^\infty ((|\na_g u|_g)^*(t))^q t^{\frac qp -1} dt = \|\na_g u\|_{p,q}^q.
\end{equation}
From the proof of Theorem \ref{PS}, we have
\begin{equation}\label{eq:u*}
u^*(t) = \int_t^\infty \frac{U(s)}{n \sigma_n \sinh^{n-1}(F(s))} ds.
\end{equation}
From the equality $\Psi(F(r)) = r$ or equivalently
\[
r = n \sigma_n \int_0^{F(r)} \sinh^{n-1} (s) ds, \quad r\geq 0,
\]
it is easy to verify that 
\begin{equation}\label{eq:estF}
n\sigma_n \sinh^{n-1}(F(r)) > (n-1)r,\quad \forall\, r > 0,
\end{equation}  
$\sinh^{n-1}(F(r)) \sim r $ as $r \to \infty$ and $\sinh^{n-1} (F(r)) \sim r^{\frac {n-1}n} $ as $r \to 0$. Consequently, we have the following estimates 
\begin{equation}\label{eq:dangdieuu*}
\lim_{r\to 0} u^*(r) r^{\frac {n-p}{np}} = \lim_{r\to \infty} u^*(r) r^{\frac 1p} = 0.
\end{equation}
It follows from \eqref{eq:gradientLorentz} and \eqref{eq:u*} that
\[
\|\na_g u\|_{p,q}^q \geq \int_0^\infty |(u^*)'(t)|^q (n \sigma_n \sinh^{n-1}(F(t)))^q t^{\frac qp -1} dt.
\]
Since $q \geq \frac{2n}{n-1}$, then we have from \cite[Lemma $2.1$]{NguyenPS2018} that 
\[
\sinh^{q(n-1)}(F(t)) \geq \lt(\frac t{\sigma_n}\rt)^{q\frac{n-1}n} + \lt(\frac{n-1}n\rt)^q \lt(\frac t{\sigma_n}\rt)^q,
\]
since $F(t) = \Phi^{-1}(t/\sigma_n)$ with $\Phi$ being defined by \cite[Formula $(2.6)$]{NguyenPS2018}. Plugging this estimate into the preceding one, we obtain
\begin{equation}\label{eq:key1}
\|\na_g u\|_{p,q}^q \geq n^q \sigma_n^{\frac qn} \int_0^\infty |(u^*)'(t)|^q t^{q\frac{n-1}n + \frac qp -1} dt + (n-1)^q \int_0^\infty |(u^*)'(t)|^q t^{q + \frac{q}p -1} dt.
\end{equation}
Notice that $v'(r) = (u^*)'(\sigma_n r^n) n \sigma_n r^{n-1}$, so by a simple change of variable $t = \sigma_n r^n$ we get
\begin{equation}\label{eq:doibien}
n^q \sigma_n^{\frac qn} \int_0^\infty |(u^*)'(t)|^q t^{q\frac{n-1}n + \frac qp -1} dt = n \sigma_n^{\frac qp} \int_0^\infty |v'(r)|^q r^{n\frac qp -1} dr.
\end{equation}
Making the change of function $u^*(t) = w(t) t^{-\frac1p}$, we have
\[
0\leq (u^*)'(t) = -w'(t) t^{-\frac1p} + \frac1p w(t) t^{-\frac 1p -1}.
\]
We can readily check that if $b -a \geq 0, b\geq 0$ and $q \geq 2$ then
\[
(b-a)^q \geq b^q + |a|^q -q a b^{q-1}.
\]
By this inequality, we have
\begin{align*}
\int_0^\infty |(u^*)'(t)|^q t^{q + \frac{q}p -1} dt &\geq \frac1{p^q} \int_0^\infty (u^*(t))^q t^{\frac qp -1} dt + \int_0^\infty |w'(t)|^q t^{q-1} dt\\
&\quad -\frac{q}{p^{q-1}} \int_0^\infty w'(t) w(t)^{q-1} dt
\end{align*}
Using integration by parts and \eqref{eq:dangdieuu*}, we get
\begin{equation}\label{eq:Hardy1dim}
\int_0^\infty |(u^*)'(t)|^q t^{q + \frac{q}p -1} dt \geq \frac1 {p^q} \int_0^\infty u^*(t)^q t^{\frac qp -1} dt + \int_0^\infty |w'(t)|^q t^{q-1} dt \geq \frac1{p^q} \|u\|_{p,q}^q.
\end{equation}
Inserting \eqref{eq:Hardy1dim} and \eqref{eq:doibien} into \eqref{eq:key1} we obtain \eqref{eq:keyestimate}. This finishes the proof of this proposition.
\end{proof}

In the proof of Theorem \ref{PoincareSobolev}, we shall need the following sharp Sobolev inequality with the fractional dimension (see \cite[Proposition $1.1$]{NguyenPLMS})

\begin{lemma}\label{Sobolevfract}
Let $\beta \geqs q \geqs 1$. Then there exists a constant $C \geqs 0$ such that
\[
\int_0^\infty |w'(r)|^q r^{\beta -1} dr \geq C \lt(\int_0^\infty |w(r)|^{\frac{\beta q}{\beta -q}} r^{\beta -1} dr\rt)^{\frac{\beta -q}{\beta}}.
\]
Furthermore, if we denote by $S(\beta,q)$ the sharp constant in the preceding inequality then equality holds if $w(r) =(1 + r^{q/(q-1)})^{-(\beta -q)/q}$.
\end{lemma}
Notice that when $\beta$ is an integer then the Lemma above is exactly the sharp Sobolev inequality of Talenti \cite{Talenti1976} and Aubin \cite{Aubin1976} applied to radially symmetric functions.

A direct computation shows that
\begin{equation}\label{eq:Sbetaq}
S(\beta,q) = \beta\lt(\frac{\beta -q}{q-1}\rt)^{q-1}\lt[\frac{q-1}q \frac{\Gamma(\frac\be q) \Gamma(\frac{\beta(q-1)}q)}{\Gamma(\beta)}\rt]^{\frac q\beta}.
\end{equation}

\section{Proof of Theorem \ref{Poincare}}
In this section, we prove Theorems \ref{Poincare}. Obviously, Theorem \ref{Poincare} is a consequence of Proposition \ref{keyprop} when $\frac{2n}{n-1} \leq q \leq p$. Here, we give the proof for any $q \leq p$. In fact, we shall use some estimates in the proof of Proposition \ref{keyprop}.

\begin{proof}[Proof of Theorem \ref{Poincare}]
It follows from the P\'olya--Szeg\"o principle from Theorem \ref{PS} that it is sufficient to prove Theorem \ref{Poincare} for radially symmetric non-increasing function $u\in W^1 L^{p,q}(\H^n)$. Let $u\in W^1L^{p,q}(\H^n)$ be such a function, and let $U$ be the function built from $\na_g u$ on the level sets of $u$. Since $q \leq p$, then the estimate \eqref{eq:gradientLorentz} holds. Hence, using integration by parts, the estimates \eqref{eq:dangdieuu*} and H\"older inequality, we have
\begin{align*}
\int_0^\infty (u^*(s))^q t^{\frac qp -1} dt & = \frac pq \int_0^\infty (u^*(s))^q d t^{\frac qp}\\
&= p \int_0^\infty (u^*(t))^{q-1} t^{\frac {q-1}p} \frac{t^{\frac 1p +1} U(t)}{n \sigma_n \sinh^{n-1} (F(t))} \frac{dt}t\\
&\leq \lt(\int_0^\infty(t^{\frac1p} u^*(s))^q \frac{dt}t \rt)^{\frac{q-1}p } \lt(\int_0^\infty (t^{\frac1p} U(t))^q \lt(\frac t{n \sigma_n \sinh^{n-1} (F(t))}\rt)^q \frac{dt}t\rt)^{\frac1q}.
\end{align*}
Whence, it holds
\begin{align}\label{eq:PSL}
\|u\|_{p,q} &\leq p\lt(\int_0^\infty (t^{\frac1p} U(t))^q \lt(\frac t{n \sigma_n \sinh^{n-1} (F(t))}\rt)^q \frac{dt}t\rt)^{\frac1q}\notag\\
& \leq \frac{p}{n-1} \lt(\int_0^\infty (t^{\frac1p} U(t))^q \frac{dt}t\rt)^{\frac1q} \notag\\
&\leq \frac p{n-1} \|\na_g u\|_{p,q}
\end{align}
here we used \eqref{eq:estF}. This proves \eqref{eq:PoincareLorentz}. 

We next check the sharpness of the constant $(n-1)^q/p^q$ in \eqref{eq:PoincareLorentz}. For $0< a < R$, let us define the function
\[
f_{a,R}(s) = \begin{cases}
a^{-\frac1p} &\mbox{if $s \in (0,a)$},\\
s^{-\frac1p} &\mbox{if $s\in [a,R)$,}\\
R^{-\frac1p} \max\{2 -s/R,0\} &\mbox{if $s \geq R$},
\end{cases}
\]
and 
\[
u_{a,R}(x) = f_{a,R}(V_g(B(0,d(0,x)))).
\]
Notice that $f_{a,R}$ is non-increasing function. Hence, by direct computations we get
\[
\|u_{a,R}\|_{p,q}^q = \int_0^\infty f_{a,R}(s)^q s^{\frac qp -1} ds = \frac pq + \ln \frac Ra + \int_1^2 (2-s)^q s^{\frac qp -1} ds.
\]
Furthermore, we have
\[
f_{a,R}'(s) = \begin{cases}
0 &\mbox{if $s \in (0,a)$ or $s \geqs 2R$},\\
-\frac1p s^{-\frac1p -1} &\mbox{if $a \leqs s \leqs R$,}\\
-R^{-\frac1p-1} &\mbox{if $R \leqs s \leqs 2R$},
\end{cases}
\]
and 
\[
|\na_g u_{a,R}(x)|_g = -f_{a,R}'(V_g(B(0,d(0,x)))) n\sigma_n \sinh^{n-1} (d(0,x)) = U_{a,R}(V_g(B(0,d(0,x))))
\]
with
\[
U_{a,R}(s) = -f_{a,R}'(s) n \sigma_n \sinh^{n-1} (F(s)).
\]
Since $F(r) \to \infty$ as $r \to \infty$, then we can check that
\[
\lim_{r\to \infty} \frac{n \sigma_n \sinh^{n-1}(F(r))}{r} = \frac1{n-1}.
\]
For any $\epsilon \geqs 0$, we can choose $a \geqs 0$ such that 
\[
n \sigma_n \sinh^{n-1}(F(r)) \leq (1 + \epsilon) (n-1) r,\quad\forall\, r\geq a.
\]
For $(2p)^{-p}R \geqs  a$, it is easy to see that $U_{a,R} \leq (n-1)(1+\ep) h(s)$ with 
\[
h(s) = \begin{cases}
a^{-\frac 1p} &\mbox{if $s \in (0,a)$ or $s \geq 2R$},\\
\frac1p s^{-\frac1p} &\mbox{if $a \leq s \leqs (2p)^{-p}R$,}\\
2R^{-\frac1p} &\mbox{if $(2p)^{-p}R \leq s \leqs 2R$,}\\
0 &\mbox{if $s \geq 2R$}.
\end{cases}
\]
Note that $h$ is non-increasing function. Hence, it holds
\begin{align*}
\|\na_g u_{a,R}\|_{p,q}^q &= \int_0^\infty (U_{a,R}^*(t))^q t^{\frac qp -1} dt\\
&\leq (n-1)^q(1+\ep)^q \int_0^\infty h(s)^q t^{\frac qp -1} dt\\
&=(n-1)^q(1+\ep)^q \lt(\frac{p}q + \frac1{p^q} \Big(\ln R -p \ln(2p) -\ln a\Big) + \frac{2^q p}q \Big(2^{\frac qp} - (2p)^{-q}\Big)\rt).
\end{align*}
Consequently, we obtain
\[
\inf_{u\in W^1L^{p,q}(\H^n), u\not\equiv 0} \frac{\|\na_g u\|_{p,q}^q}{\|u\|_{p,q}^q} \leq \limsup_{R\to \infty} \frac{\|\na_g u_{a,R}\|_{p,q}^q}{\|u_{a,R}\|_{p,q}^q} \leq (1+ \epsilon)^q\frac{(n-1)^q}{p^q}.
\]
Since $\epsilon \geqs 0$ is arbitrary, then we get
\[
\inf_{u\in W^1L^{p,q}(\H^n), u\not\equiv 0} \frac{\|\na_g u\|_{p,q}^q}{\|u\|_{p,q}^q} \leq \frac{(n-1)^q}{p^q}.
\]
This shows that $(n-1)^q/p^q$ is the sharp constant in \eqref{eq:PoincareLorentz}.

From the estimate \eqref{eq:PSL}, we see that if $u\not\equiv 0$ then the second inequality is strict since $ n\sigma_n \sinh^{n-1}(F(t)) > (n-1) t$ for $t > 0$. Consequently, the constant $(n-1)^q/p^q$ in \eqref{eq:PoincareLorentz} is not attained by a non-zero function $u$.
\end{proof}

\section{Proof of Theorem \ref{PoincareSobolev}}
This section is addressed to prove Theorem \ref{PoincareSobolev}. The proof is based on Theorem \ref{PS}, Proposition \ref{keyprop} and the weighted Sobolev inequality.

\begin{proof}[Proof of Theorem \ref{PoincareSobolev}]
By the P\'olya--Szeg\"o principle from Theorem \ref{PS}, it is enough to prove Theorem \ref{PoincareSobolev} for radially symmetric non-increasing functions $u \in W^1 L^{p,q}(\H^n)$. For $r \geq 0$, we define $v(r) = u^*(\sigma_n r^n)$. Since $\frac{2n}{n-1} \leq q \leq p$, from Proposition \ref{keyprop} we have
\begin{equation}\label{eq:a1}
\|\na_g u\|_{p,q}^q - \lt(\frac{n-1}p\rt)^q \|u\|_{p,q}^q \geq n \sigma_n^{\frac qp}\int_0^\infty |v'(r)|^q r^{\frac{nq}p -1} dr.
\end{equation}
Making the change of function $v(r) =w(r) r^{-\frac{n-p}p}$, we have
\[
0\leq -v'(r) = -w'(r) r^{-\frac{n-p}p} + \frac{n-p}p w(r) r^{-\frac{n-p}p -1}.
\]
Moreover, it follows from \eqref{eq:dangdieuu*} that
\[
\lim_{r\to 0} w(r) = \lim_{r\to\infty} w(r) =0.
\]
By the convexity of function $t \to |t|^q$, the integration by parts and the asymptotic behavior of $w$ above, we have
\begin{align*}
\int_0^\infty |v'(r)|^q r^{\frac{nq}p -1} dr&\geq \lt(\frac{n-p}p\rt)^q \int_0^\infty v(r)^q r^{n\frac{q}p -q -1} dr + \lt(\frac{n-p}p\rt)^{q-1} \int_0^\infty (w(r)^q)' dr\\
&=\lt(\frac{n-p}p\rt)^q \frac1{n \sigma_n^{\frac q{p^*}}} \|u\|_{p^*,q}^q.
\end{align*}
Consequently, we obtain
\[
\|\na_g u\|_{p,q}^q - \lt(\frac{n-1}p\rt)^q \|u\|_{p,q}^q \geq \lt(\frac{n-p}p\rt)^q \sigma_n^{\frac qn} \|u\|_{p^*,q}^q.
\]
This proves \eqref{eq:PSLorentz1} for $l = q$.

For $q\leqs l \leq \frac{nq}{n-p}$, we let $\alpha = \frac{nq -l(n-p)}p \in [0,q)$. Making the change of function $v(r) = w(r^{\frac{q-\al}q})$, we have
\begin{equation}\label{eq:a2}
\int_0^\infty |v'(r)|^q r^{\frac{nq}p -1} dr = \lt(\frac{q-\al}q\rt)^{q-1} \int_0^\infty |w'(r)|^q r^{\frac{(nq -p\al)q}{p(q-\al)} -1} dr.
\end{equation}
Notice that $\frac{(nq -p\al)q}{p(q-\al)} \geqs q$ since $p \leqs n$. Lemma \ref{Sobolevfract} implies
\begin{align*}
\int_0^\infty |w'(r)|^q r^{\frac{(nq -p\al)q}{p(q-\al)} -1} dr&\geq S\Big(\frac{(nq -p\al)q}{p(q-\al)},q\Big) \lt(\int_0^\infty w(r)^{\frac{nq -p\al}{n-p}} r^{\frac{(nq -p\al)q}{p(q-\al)} -1} dr\rt)^{\frac{q(n-p)}{nq -p\al}}\\
&=S\Big(\frac{lq}{l-p},q\Big) \lt(\int_0^\infty w(s)^l s^{\frac{nl}{p^*} \frac q{q-\al} -1} ds\rt)^{\frac ql}\\
&= S\Big(\frac{lq}{l-p},q\Big) \Big(\frac{(n-p)(l-q)}{qp}\Big)^{\frac ql} \lt(\int_0^\infty v(s)^l s^{\frac{nl}{p^*} -1} ds\rt)^{\frac ql}\\
&= S\Big(\frac{lq}{l-p},q\Big) \Big(\frac{(n-p)(l-q)}{qp}\Big)^{\frac ql} n^{-\frac ql} \sigma_n^{-\frac q{p^*}} \lt(\int_0^\infty (u^*(t))^l t^{\frac l{p^*} -1} dt\rt)^{\frac ql}\\
&= S\Big(\frac{lq}{l-p},q\Big) \Big(\frac{(n-p)(l-q)}{qp}\Big)^{\frac ql} n^{-\frac ql} \sigma_n^{-\frac q{p^*}} \|u\|_{p^*,l}^q,
\end{align*}
here we make the change of variable $r =s^{\frac{q-\al}q}$ in the first equality, and $t = \sigma_n s^n$ in the second equality. Combining the previous estimate together with \eqref{eq:a1} and \eqref{eq:a2} yields
\begin{align*}
\|\na_g u\|_{p,q}^q - \lt(\frac{n-1}p\rt)^q \|u\|_{p,q}^q \geq n^{1-\frac ql} \sigma_n^{\frac qn} \Big(\frac{(n-p)(l-q)}{qp}\Big)^{q+\frac ql -1} S\Big(\frac{lq}{l-p},q\Big) \|u\|_{p^*,l}^q
\end{align*}
as wanted. This proves the inequality \eqref{eq:PSLorentz1}.

\end{proof}

\section{Proof of Theorem \ref{3rdMT}}
In this section, we prove Theorem \ref{3rdMT}. The proof is based on Proposition \ref{keyprop} and the Moser--Trudinger inequality involving to the fractional dimension in Lemma \ref{MT} below. Let $\theta \geqs 1$, we denote by $\lam_\theta$ the measure on $[0,\infty)$ of density
\[
d\lam_{\theta} = \theta \sigma_{\theta} x^{\theta -1} dx, \quad \sigma_{\theta} = \frac{ \pi^{\frac\theta 2}}{\Gamma(\frac\theta 2+ 1)}.
\]
For $0\leqs R \leq \infty$ and $1\leq p \leqs \infty$, we denote by $L_\theta^p(0,R)$  the weighted Lebesgue space of all measurable functions $u: (0,R) \to \R$ for which
\[
\|u\|_{L^p_\theta(0,R)}= \lt(\int_0^R |u|^p d\lam_\theta\rt)^{\frac1p} \leqs \infty.
\]
Besides, we define
\[
W^{1,p}_{\al,\theta}(0,R) =\Big\{u\in L^p_\theta(0,R)\, :\, u' \in L_\alpha^p(0,R),\,\, \lim_{x\to R^{-}} u(x) =0\Big\}, \quad \al, \theta \geqs 1.
\]
In \cite{deOliveira}, de Oliveira and do \'O prove the following sharp Moser--Trudinger inequality involving the measure $\lam_\theta$: suppose $0 \leqs R \leqs \infty$ and $\alpha \geq 2, \theta \geq 1$, then 
\begin{equation}\label{eq:MTOO}
D_{\al,\theta}(R) :=\sup_{u\in W^{1,\al}_{\al,\theta}(0,R),\, \|u'\|_{L^\alpha_\alpha(0,R)} \leq 1} \int_0^R e^{\mu_{\al,n} |u|^{\frac{\alpha}{\alpha -1}}} d\lam_\theta \leqs \infty
\end{equation}
where $\mu_{\al,\theta} = \theta \alpha^{\frac1{\al -1}} \sigma_{\al}^{\frac1{\alpha -1}}$. Denote $D_{\al,\theta} = D_{\al,\theta}(1)$. It is easy to see that $D_{\al,\theta}(R) = D_{\al,\theta} R^\theta$. We shall need the following result. 
\begin{lemma}\label{MT}
Let $2\leq \alpha \leq n$. There exists a constant $C_{\alpha,n} \geqs 0$ such that for any $u \in W^{1,\alpha}_{\alpha,\alpha}(0,\infty)\cap L_n^n(0,\infty), $ $u' \leq 0$ and $\|u\|_{L^\alpha_\alpha(0,\infty)}^\alpha + \|u'\|_{L^\alpha_\alpha(0,\infty)}^\alpha\leq 1$, it holds
\begin{equation}\label{eq:Abreu}
\int_0^\infty \Phi_\alpha(\mu_{\al,n} |u|^{\frac \alpha{\alpha -1}}) d\lam_n \leq C_{\al,n}\big(1 + \|u\|_{L^n_n(0,\infty)}^n\big).
\end{equation}
\end{lemma}
\begin{proof}
We follow the argument in \cite{Ruf2005} by using the inequality \eqref{eq:MTOO}. Since $u'\leq 0$ then $u$ is non-increasing function. So, for any $r \geqs 0$, we have
\begin{equation}\label{eq:pointwise}
u(r)^\alpha \leq \frac 1{\si_\al r^n}\int_0^r u(s)^\alpha d\lam_\alpha \leq \frac {\int_0^\infty u(s)^\alpha d\lam_\alpha}{\si_\al r^\alpha}.
\end{equation}
Fix a $R \geqs 0$, we define the function $w$ by $w(r) =u(r) -u(R)$ if $r \leqs R$ and $w(r) =0$ if $r \geqs R$. Then, $w \in W^{1,\alpha}_{\al,\al}(0,R)$ and 
\begin{equation}\label{eq:tach}
\|w\|_{L^\al_\al(0,R)}^\al = \int_0^R |u'(s)|^\al d\lam_\alpha \leq 1 - \int_0^\infty |u|^\alpha d\lam_\alpha.
\end{equation}
For $r \leq R$, we have $u(r) = w(r) + u(R)$. Since $\alpha \geq 2$, then there exists a constant $C_\al \geqs 0$ depending only $\alpha$ such that
\[
u(r)^{\frac\alpha {\al -1}}\leq w(r)^{\frac\alpha {\al -1}} + C w(r)^{\frac1{\alpha -1}} u(R) + u(R)^{\frac\alpha {\al -1}}.
\]
By Young's inequality and \eqref{eq:pointwise}, we obtain
\begin{align}\label{eq:b1}
u(r)^{\frac\alpha {\al -1}} &\leq w(r)^{\frac\alpha {\al -1}} \lt(1 + \frac{C}{\al} u(R)^\alpha\rt) + \frac{\al-1}\al  + u(R)^{\frac\alpha {\al -1}}\notag\\
&\leq w(r)^{\frac\alpha {\al -1}} \lt(1 + \frac{C\int_0^\infty u(s)^\alpha d\lam_\alpha}{\al \si_\al R^\al}\rt) + \frac{\al-1}{\alpha} + \lt(\frac1{\si_\al R^\al}\rt)^{\frac1{\al-1}}.
\end{align}
Choosing $R \geq 1$ large enough such that $\frac{C}{\al \si_\al R^\al} \leq 1$, and set 
\[
v(r) = w(r) \lt(1 + \frac{C\int_0^\infty u(s)^\alpha d\lam_\alpha}{\al \si_\al R^\al}\rt)^{\frac{\al-1}\al}.
\]
We have $v \in W^{1,\alpha}_{\al,\al}(0,R)$. Furthermore, there exists a constant $C_\al \geqs 1$ depending only on $\alpha$ such that $(1+ t)^{\alpha -1} \leq 1+ C_\al t$ for any $t\in [0,1]$. Hence, by \eqref{eq:tach} we have
\begin{align*}
\|v\|_{L^\al_\al(0,R)}^\al& = \|w\|_{L^\al_\al(0,R)}^\al\lt(1 + \frac{C}{\al \si_\al R^\al}\rt)^{\alpha-1}\\
&\leq \lt(1 - \int_0^\infty u(s)^\alpha d\lam_\alpha\rt)\lt(1 + \frac{C\int_0^\infty u(s)^\alpha d\lam_\alpha}{\al \si_\al R^\al}\rt)^{\al -1}\\
&\leq \lt(1 - \int_0^\infty u(s)^\alpha d\lam_\alpha\rt)\lt(1 + \frac{C C_\al\int_0^\infty u(s)^\alpha d\lam_\alpha}{\al \si_\al R^\al}\rt)\\
&\leq 1 - \lt(1-\frac{C C_\al}{\al \si_\al R^\al}\rt)\int_0^\infty u(s)^\alpha d\lam_\alpha.
\end{align*}
We can choose a $R \geq 1$ large enough and depending only on $\al$ such that $1-\frac{C C_\al}{\al \si_\al R^\al} \geq 0$. For such a $R$, we have by \eqref{eq:MTOO} that
\begin{equation}\label{eq:lowR}
\int_0^R e^{\mu_{\al,n} v(s)^{\frac \al{\al-1}}} d\lam_n \leq D_{\al,n} R^n.
\end{equation}
For $r \geq R$, by \eqref{eq:pointwise} we have $u(r) \leq (\si_\al)^{-\frac1\al} R^{-1}$ so there is a constant $C(\al,n)$ depending only on $n$ and $\alpha$ such that
\[
\Phi_\al(\mu_{\al, n} u(r)^{\frac\al{\al-1}}) \leq C(\al,n) u(r)^n,
\]
here we use $(j_\al-1) \frac \al{\al-1} \geq n$ and the fact that $u$ is bounded from above by a constant depending only on $\alpha$. Consequently, we have
\begin{equation}\label{eq:upR}
\int_R^\infty \Phi_\al(\mu_{\al, n} u(r)^{\frac\al{\al-1}}) d\lam_n \leq C(\alpha,n) \|u\|_{L^n_n(0,\infty)}^n.
\end{equation}
Putting \eqref{eq:b1}, \eqref{eq:lowR}, \eqref{eq:upR} and the fact $R\geq 1$ together, we obtain
\begin{align*}
\int_0^\infty \Phi_\al(\mu_{\al, n} u(r)^{\frac\al{\al-1}}) d\lam_n&\leq \int_0^R e^{\mu_{\al, n} u(r)^{\frac\al{\al-1}}} d\lam_n + \int_R^\infty \Phi_\al(\mu_{\al, n} u(r)^{\frac\al{\al-1}}) d\lam_n\\
&\leq \int_0^R e^{\mu_{\al,n} v(r)^{\frac\al{\al-1}} + \mu_{\al,n}\lt(\frac{\al-1}\al + \si_\al^{-1/(\al-1)}\rt)} d\lam_n + C(\alpha,n) \|u\|_{L^n_n(0,\infty)}^n\\
&\leq e^{\mu_{\al,n}\lt(\frac{\al-1}\al + \si_\al^{-1/(\al-1)}\rt)} D_{\al,n} R^n + C(\alpha,n) \|u\|_{L^n_n(0,\infty)}^n\\
&\leq C_{\al,n}\big(1 + \|u\|_{L^n_n(0,\infty)}^n\big),
\end{align*}
for some constant $C_{\al,n} \geqs 0$ depending only on $n$ and $\alpha$.
\end{proof}

For any $\tau \geqs 0$ and $u \in W^{1,\alpha}_{\alpha,\alpha}(0,\infty)\cap L_n^n(0,\infty),$ such that $u' \leq 0$ and $\tau \|u\|_{L^\alpha_\alpha(0,\infty)}^\alpha + \|u'\|_{L^\alpha_\alpha(0,\infty)}^\alpha\leq 1$. Applying the inequality \eqref{eq:Abreu} to function $u_\tau(x) = u(\tau^{-\frac1\alpha} x)$ and making the change of variables, we get
\begin{equation}\label{eq:Abreu1}
 \int_0^\infty \Phi_\alpha(\mu_{\al,n} |u|^{\frac \alpha{\alpha -1}}) d\lam_n \leq C_{\al,n}\lt(\frac{1}{\tau^{\frac n\alpha}} + \|u\|_{L^n_n(0,\infty)}^n\rt).
\end{equation}

We are now ready to prove Theorem \ref{3rdMT} by using Proposition \ref{keyprop} and the Moser--Trudinger inequality \eqref{eq:Abreu1}.

\begin{proof}[Proof of Theorem \ref{3rdMT}]
By the P\'olya--Szeg\"o principle from Theorem \ref{PS}, it is enough to prove Theorem \ref{3rdMT} for radially symmetric non-increasing function $u \in W^1 L^{n,q}(\H^n)$, i.e., we will prove the existence of a constant $C_{n,q,\lam} >0$ such that 
\begin{equation}\label{eq:need}
\int_{\H^n} \Phi_q(\al_{n,q} u^{\frac q{q-1}}) dV_g \leq C_{n,q,\lam},
\end{equation}
for any radially symmetric, non-increasing function $u\in W^{1} L^{n,q}(\H^n)$ satisfying 
$$\|\na_g u\|_{n,q}^q -\lam \|u\|_{n,q}^q\leq 1.$$ 
Let $u$ be such a function, we define $v(r) =u^*(\sigma_n r^n)$, we then have $\lim_{r\to \infty} v(r) =0$. Since $n\geq 4$ and $\frac{2n}{n-1} \leq q \leq n$, then by Proposition \ref{keyprop} we have
\[
\|\na_g u\|_{n,q}^q - \lt(\frac{n-1}n\rt)^q \|u\|_{n,q}^q \geq n \sigma_n^{\frac qn}\int_0^\infty |v'(r)|^q r^{q -1} dr
\]
For $\lam \leqs \big(\frac{n-1}n\big)^q$, denote $\tau = \big(\frac{n-1}n\big)^q -\lam \geqs 0$. So, we have
\[
1\geq \|\na_g u\|_{n,q}^q - \lam \|u\|_{n,q}^q \geq n \sigma_n^{\frac qn}\int_0^\infty |v'(r)|^q r^{q -1} dr + \tau \|u\|_{n,q}^q.
\]
We have
\[
\|u\|_{n,q}^q =\int_0^\infty (u^*(t))^q t^{\frac qn -1} dt = n \sigma_n^{\frac qn} \int_0^\infty v(r)^q r^{q -1} dr.
\]
Thus, it holds
\[
n \si_n^{\frac qn} \lt(\int_0^\infty |v'(r)|^q r^{q -1} dr + \tau \int_0^\infty v(r)^q r^{q -1} dr\rt)\leq 1.
\]
Set 
\begin{equation}\label{eq:doiham}
w(r) = \lt(\frac{n \si_n^{\frac qn}}{q \si_q}\rt)^{\frac1q} v(r),
\end{equation}
we then have $w \in W^{1,q}_q(0,\infty)$ and $\|w'\|_{L^q_q(0,\infty)}^q + \tau \|w\|_{L^q_q(0,\infty)}^q \leq 1$. Applying \eqref{eq:Abreu1}, we get
\begin{equation}\label{eq:c1}
\int_0^\infty \Phi_q(\mu_{q,n} |w|^{\frac q{q -1}}) d\lam_n \leq C_{q,n}\lt(\frac{1}{ \tau^{\frac nq}} + \|w\|_{L^n_n(0,\infty)}^n\rt).
\end{equation}
Since $q \leq n$, then by \cite[Lemma $3.2$]{YangLi2019} there exists a constant $C >0$ depending on $n$ and $q$ such that
\[
\Big(\int_{\H^n} |u|^n dV_g\Big)^{\frac qn} \leq C \|\na_g u\|_{n,q}^q \leq C\frac{\big(\frac{n-1}n\big)^q}{\big(\frac{n-1}n\big)^q -\lam} (\|\na_g u\|_{n,q}^q -\lam \|u\|_{n,q}^q) \leq C\frac{\big(\frac{n-1}n\big)^q}{\big(\frac{n-1}n\big)^q -\lam},
\]
which implies
\[
\|v\|_{L^n_n(0,\infty)}^n = \int_0^\infty (u^*(t))^n dt =\int_{\H^n} |u|^n dV_g  \leq \lt(C\frac{\big(\frac{n-1}n\big)^q}{\big(\frac{n-1}n\big)^q -\lam}\rt)^{\frac nq}.
\]
Consequently, there holds
\[
\|w\|_{L^n_n(0,\infty)}^n \leq \lt(C\frac{n \si_n^{\frac qn}}{q \si_q}\frac{\big(\frac{n-1}n\big)^q}{\big(\frac{n-1}n\big)^q -\lam}\rt)^{\frac nq}.
\]
Inserting this estimate and \eqref{eq:doiham} into \eqref{eq:c1}, we obtain
\[
 \int_0^\infty \Phi_q(\al_{n,q} |v|^{\frac q{q -1}}) d\lam_n \leq \frac{\tilde C_{n,q}}{\tau^{\frac nq}},
 \]
for some constant $\tilde C_{n,q} >0$ depending on $n$ and $q$, here we use $\mu_{q,n} = n q^{\frac 1{q-1}} \si_q^{\frac1{q-1}}$. In other hand, we have
\[
\int_{\H^n} \Phi_q(\al_{n,q} |u|^{\frac q{q-1}}) dV_g = \int_0^\infty \Phi_q(\al_{n,q} (u^*(t))^{\frac q{q-1}}) dt = \int_0^\infty \Phi_q\big(\al_{n,q} |v|^{\frac q{q -1}}\big) d\lam_n.
\]
Therefore, we have
\[
\int_{\H^n} \Phi_q(\al_{n,q} |u|^{\frac q{q-1}}) dV_g \leq \tilde C_{n,q}\lt(\Big(\frac{n-1}n\Big)^q - \lam\rt)^{-\frac nq}.
\]
This proves the inequality \eqref{eq:need}. The proof of Theorem \ref{3rdMT} is then completely finished.
\end{proof}


\bibliographystyle{abbrv}

\end{document}